\documentclass[12pt]{article}

\usepackage{amsmath,amsthm,amssymb,amsfonts}
\usepackage{caption}
\usepackage{latexsym}
\usepackage{enumerate}
\usepackage{graphicx}
\usepackage{color}
\usepackage{fullpage}

\newtheorem{thm}{Theorem}
\newtheorem{lem}[thm]{Lemma}

\newtheorem{cor}[thm]{Corollary}

\newtheorem{definition}[thm]{Definition}

\begin{document}

\title{The 1/3-2/3 Conjecture for ordered sets whose cover graph is a forest}
\author{Imed Zaguia  \\
\small Department of Mathematics \& Computer Science\\ \small  Royal Military College of Canada\\
\small P.O.Box 17000, Station Forces, K7K 7B4 Kingston, Ontario, CANADA \\
{\small E-mail: imed.zaguia@rmc.ca} }

\date{\today}
\maketitle

\begin{abstract} A balanced pair in an ordered set
$P=(V,\leq)$ is a pair $(x,y)$ of elements of $V$ such that the
proportion of linear extensions of $P$ that put $x$ before $y$ is
in the real interval $[1/3, 2/3]$. We define the notion of a good pair and claim any ordered set that has a good pair will satisfy the conjecture and furthermore every ordered set which is not totally ordered and has a forest as its cover graph has a good pair.

\end{abstract}

\section{Introduction}

Throughout, $P =(V, \leq)$ denotes a \emph{finite ordered set}, that
is, a \emph{finite} set $V$ and a binary relation $\leq$ on $V$ which is
reflexive, antisymmetric and transitive. A \emph{linear extension} of $P =(V, \leq)$ is a total ordering $\preceq$ of $V$ which extends $\leq$, i.e. such that for every $x,y\in V$, $x\preceq y$ whenever $x \leq y$.

For a pair $(x,y)$ of elements of $V$ we denote by $\mathbb{P}(x\prec y)$ the proportion of linear extensions of $P$
that put $x$ before $y$. Call a pair $(x,y)$ of elements of $V$ a \emph{balanced pair} in $P=(V,\leq)$
if $1/3 \leq \mathbb{P}(x\prec y) \leq 2/3$. The 1/3-2/3 Conjecture states that every finite ordered set which is not totally ordered has a balanced pair. If true, the example (a) depicted in Figure \ref{alpha} would show that the result is best
possible. The 1/3-2/3 Conjecture first appeared in a paper of Kislitsyn \cite{ki}. It was also formulated independently by Fredman in about 1975 and again by Linial \cite{li}.

The 1/3-2/3 Conjecture is known to be true for ordered sets with a nontrivial automorphism \cite{ghp},
for ordered sets of width two \cite{li}, for semiorders \cite{gb1}, for bipartite ordered sets
\cite{tgf}, for 5-thin posets \cite{gb2}, and for 6-thin posets \cite{pm}. See \cite{gb} for a
survey.

Recently, the author proved that the 1/3-2/3 Conjecture is true for ordered sets having no $N$ in their Hasse diagram \cite{za}. Using similar ideas we prove that the 1/3-2/3 Conjecture is true for ordered sets whose cover graph is a forest.\\

Let $P=(V,\leq)$ be an ordered set. For $x,y\in V$ we say that $y$ is an \emph{upper cover} of $x$ or that $x$ is
\emph{a lower cover} of $y$ if $x<y$ and there is no element $z\in V$ such that $x<z<y$. Also, we say that
$x$ and $y$ are \emph{comparable} if $x\leq y$ or $y\leq x$ and we set $x\sim y$; otherwise we say that $x$ and $y$ are
\emph{incomparable} and we set $x\nsim y$. We denote by $inc(P)$ the set of incomparable pairs of $P$, that is,
$inc(P):=\{(x,y): x \nsim y\}$. A \emph{chain} is a totally ordered set. For an element $u\in V$, set $D(u):=\{v\in V : v< u\}$ and $U(u):=\{v\in V : u< v\}$. The \emph{dual} of $P$, denoted by $P^*$, is the order defined on
$V$ as follows: $x\leq y$ in $P^*$ if and only if $y\leq x$ in $P$.



\begin{definition}\label{goodpair}
Let $P$ be an ordered set. A pair $(a,b)$ of elements of $V$ is \emph{good} if the following two conditions hold simultaneously in $P$ or in its dual.
\begin{enumerate}[(i)]
\item $D(a)\subseteq D(b)$ and $U(b)\setminus U(a)$ is a chain (possibly empty); \underline{and}
\item $\mathbb{P}(a\prec b)\leq \frac{1}{2}$.
\end{enumerate}
\end{definition}

We notice at once that if $(a,b)$ is a good pair, then $a$ and $b$ are necessarily incomparable. 

The relation between good pairs and balanced pairs is stated in the following theorem.

\begin{thm}\label{goodisbalanced}A finite ordered set that has a good pair has a balanced pair.
\end{thm}
We prove Theorem \ref{goodisbalanced} in Section 2.

A good pair is not necessarily a balanced pair (for an example consider the pair $(y,t)$ in example (c) Figure \ref{alpha}). The following theorem gives instances of good pairs that are balanced pairs. Before stating our next result we first need a definition. Let
$P=(V,\leq)$ be an ordered set. A subset $A$ of $V$ is called {\it autonomous} (or an interval or a module or a clan) in $P$ if for all $v\not\in A$
and for all $a,a^{\prime} \in A$

\begin{equation}
(v<a\Rightarrow v < a^{\prime})\;\mathrm{and}\;(a<v\Rightarrow
a^{\prime} < v).
\end{equation}



\begin{thm}\label{goodeqbalanced} Let $P=(V,\leq)$ be an ordered set and let $(x,y) \in inc(P)$. Suppose that
one of the following propositions holds for $P$ or for its dual.
\begin{enumerate}[(i)]
\item There exists $z\in V$ such that $x<z$, $x\nsim y \nsim z$ and $\{x,y\}$ is autonomous in $P\setminus \{z\}$ (see example (a) Figure~\ref{alpha}).
\item There are $z,t\in V$ such that $x<z$, $y<t$, $y\nsim z$, $x\nsim t$ and $\{x,y\}$ is
autonomous for $P\setminus \{z,t\}$ (see example (b) Figure~\ref{alpha}).
\item There are $z,t\in V$ such that $t<x<z$, $y$ is incomparable to
both $t$ and $z$, and $\{x,y\}$ is autonomous for $P\setminus
\{z,t\}$ (see example (c) Figure~\ref{alpha}).
\end{enumerate}
Then $(x,y)$ is balanced in $P$.
\end{thm}
We prove Theorem \ref{goodeqbalanced} in Section 3.

\begin{figure}[h]
\begin{center}
\includegraphics[width=250pt]{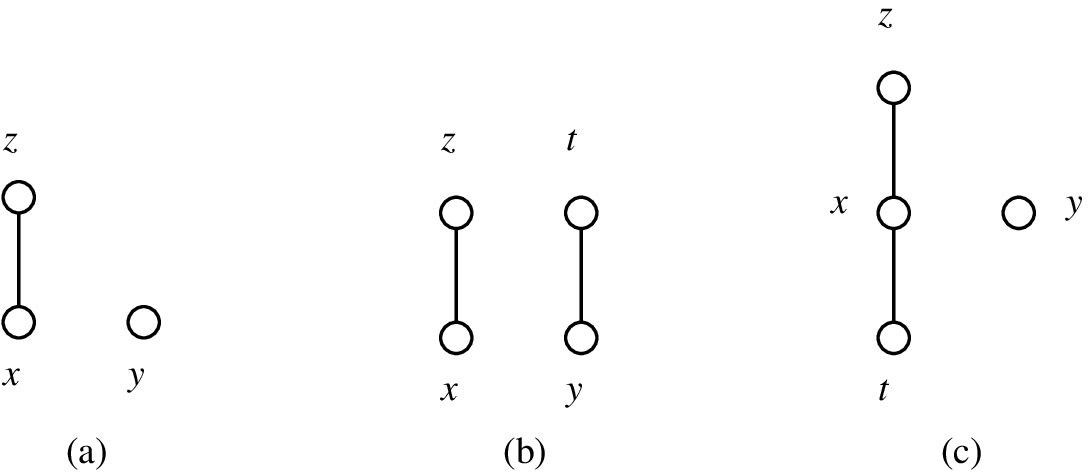}
\end{center}
\caption{}
\label{alpha}
\end{figure}

A \emph{semiorder} is an order which does not contain the orders depicted in Figure 1 (b) and 1 (c). Brightwell \cite{gb1} proved that every semiorder has a pair $(x,y)$ satisfying condition (i) of Theorem \ref{goodeqbalanced} and that either the pair $(x, y)$ is balanced, or $\mathbb{P}_{P}(x\prec z\prec y)>\frac{1}{3}$. Theorem \ref{goodeqbalanced} shows that the former always occurs. As a result we obtain this.

\begin{cor}\label{semiorder2} A balanced pair in a semiorder can be found in polynomial time.
\end{cor}

The next definition describes a particular instance of a good pair.

\begin{definition}\label{def:02} Let $P$ be an ordered set. A pair $(a,b)$ of elements of $V$ is \emph{very good} if the following two conditions hold simultaneously in $P$ or in its dual.
\begin{enumerate}[(i)]
\item $D(a)= D(b)$; \underline{and}
\item $U(a)\setminus U(b)$ and $U(b)\setminus U(a)$ are chains (possibly empty).
\end{enumerate}
\end{definition}

For instance, the pairs $(x,y)$ and $(z,y)$ in example (a) Figure~\ref{alpha} are very good . So are the pairs $(x,y)$ and $(z,t)$ in example (b) Figure~\ref{alpha}. Also, the pairs $(t,y)$ and $(y,z)$ in example (c) Figure~\ref{alpha} are very good . Observe that every ordered set of width two has a very good pair. We have already mentioned that a semiorder which is not totally ordered has a very good pair. In \cite{za}, the author proved that every $N$-free ordered set which is not totally ordered has a very good pair. We now present another instance of a class of ordered sets that have a very good pair.

The \emph{cover graph} of an ordered set $P=(V,\leq)$ is the graph $Cov(P)=(V,E)$ such that $\{x,y\}\in E$ if and only if $x$ covers $y$ in $P$.

\begin{thm}\label{forest} Let $P$ be an ordered set not totally ordered
whose cover graph is a forest. Then $P$ has a very good pair, and hence has a balanced pair.
\end{thm}
Section 4 is devoted to the proof of Theorem \ref{forest}.

We mention that an algorithm requiring $O(n^2)$ arithmetic operations for computing the number of linear extensions of an ordered set whose cover graph is a tree was given in \cite{atkinson}.

\section{Proof of Theorem \ref{goodisbalanced}}

We recall that an incomparable pair $(x,y)$ of elements is {\it critical} if $U(y)\subseteq
U(x)$ and $D(x)\subseteq~D(y)$. The set of critical pairs of $P$ is denoted by $crit(P)$.

\begin{lem} \label{l3} Suppose $(x,y)$ is a critical pair in $P$
and consider any linear extension of $P$ in which $y\prec x$. Then the
linear order obtained by swapping the positions of $y$ and $x$ is
also a linear extension of $P$. Moreover, $\mathbb{P}(x\prec y) \geq
\frac{1}{2}$.
\end{lem}
\begin{proof}Let $L$ be a linear extension that puts $y$ before $x$ and let
$z$ be such that $y\prec z\prec x$ in $L$. Then $z$ is incomparable with both $x$ and $y$ since $(x,y)$ is
a critical pair in $P$. Therefore, the linear order $L'$ obtained by swapping $x$
and $y$, that is $L'$ puts $x$ before $y$, is a linear extension of
$P$. Then map $L \mapsto L'$ from the set of linear extensions that
put $y$ before $x$ into the set of linear extensions that put $x$
before $y$ is clearly one-to-one. Hence, $\mathbb{P}(y\prec x) \leq \mathbb{P}(x\prec y)$
and therefore $\mathbb{P}(x\prec y) \geq \frac{1}{2}$.
\end{proof}

We now prove Theorem \ref{goodisbalanced}.

\begin{proof} We prove the theorem by contradiction. Let $P=(V,\leq)$ be an ordered set having a good pair $(a,b)$. We assume that $P$ has no balanced pair and we argue to a contradiction.

Then $U(b)\setminus U(a) \neq \varnothing$ because otherwise $(a,b)$ is a critical pair and hence $\mathbb{P}(a\prec b)\geq \frac{1}{2}$ (Lemma \ref{l3}). Since $(a,b)$ is a good pair $\mathbb{P}(a\prec b)\leq \frac{1}{2}$ and hence $\mathbb{P}(a\prec b)= \frac{1}{2}$ and therefore $(a,b)$ is balanced which is impossible by assumption.

Say $[U(b)\setminus U(a)]\cup \{b\}$ is the chain $b=b_1<\cdots <b_n$. Then
\[\mathbb{P}(a\prec b_1)<\frac{1}{3}.\]

Define now the following quantities
\begin{eqnarray*}
q_1 &=& \mathbb{P}(a\prec b_1),\\
q_j &=& \mathbb{P}(b_{j-1}\prec a \prec b_j)(2\leq j\leq n),\\
q_{n+1} &=& \mathbb{P}(b_n\prec a).
\end{eqnarray*}

\noindent\textbf{Lemma.}\cite{za} The real numbers $q_j$ ($1\leq j\leq n+1$) satisfy:
\begin{enumerate}[(i)]
\item $0\leq q_{n+1}\leq \cdots \leq q_1< \frac{1}{3},$
\item $\sum_{j=1}^{n+1}q_j=1.$
\end{enumerate}
\begin{proof}Since $q_1,\cdots,q_{n+1}$ is a probability distribution, all we have to show is that
$q_{n+1}\leq \cdots \leq q_1$. To show this we exhibit a one-to-one mapping from the event whose probability
is $q_{j+1}$ into the event with probability $q_j$ ($1\leq j\leq n$). Notice that in a linear extension for
which $b_j\prec a\prec b_{j+1}$ every element $z$ between $b_j$ and $a$ is incomparable to both $b_j$ and $a$. Indeed, such an element $z$ cannot be comparable to $b_j$ because otherwise $b_j<z$ in $P$ but the only
element above $b_j$ is $b_{j+1}$ which is above $a$ in the linear extension. Now $z$ cannot be comparable to $a$ as well because otherwise $z<a$ in $P$ and hence $z<b=b_1<b_j$ (by assumption we have that $D(a)\subseteq D(b)$).
The mapping from those linear extensions in which $b_j\prec a\prec b_{j+1}$ to those
in which $b_{j-1}\prec a\prec b_{j}$ is obtained by swapping the positions of $a$ and $b_j$.
This mapping clearly is well-defined and one-to-one. This completes the proof of the lemma.
\end{proof}

Theorem \ref{goodisbalanced} can be proved now: let $r$ be defined by
\[\sum_{j=1}^{r-1}q_j\leq \frac{1}{2}<\sum_{j=1}^{r}q_j\]
Since $\sum_{j=1}^{r-1}q_j=\mathbb{P}(a\prec b_{r-1})\leq \frac{1}{2}$, it follows that $\sum_{j=1}^{r-1}q_j<\frac{1}{3}$.
Similarly $\sum_{j=1}^{r}q_j=\mathbb{P}(a\prec b_{r})$ must be $>\frac{2}{3}$. Therefore $q_r>\frac{1}{3}$, but this contradicts $\frac{1}{3}>q_1\geq q_r$.
\end{proof}

\section{Proof of Theorem \ref{goodeqbalanced}}

Let $P=(V,\leq)$ be an ordered set. Denote by $Ext(P)$ the set of all \emph{extensions} of $P$ (or \emph{refinements} of the order defined on $P$), that is, all orders $\preceq$ on $V$ in which $x\preceq y$ whenever $x\leq y$ in $P$. Then $Ext(P)$ is itself ordered: for $Q,R\in
Ext(P), Q<R$ if $R$ itself is an extension of $Q$. For every pair
$(a,b)\in V\times V$, the \emph{transitive closure} of $P\cup
\{(a,b)\}$, denoted by $P\vee (a,b)$, is $P\cup \{(x,y) : x \leq a
\; {\rm and}\; b\leq y\; {\rm in}\; P\}$. As it is well-known, if
$(b,a)\not \in P$ then this is an order. It is shown in \cite{dk}
that if $Q$ and $R$ are elements of $Ext(P)$ then $R$ covers $Q$ in
$Ext(P)$ if and only if $R$ is obtained from $Q$ by adding the
comparability $a<b$ corresponding to a critical pair $(a,b)$ of $Q$.
In this case $R=Q\vee\{(a,b)\}=Q\cup \{(a,b)\}$. It turns out that the maximal elements of $Ext(P)$ are the linear extensions of $P$ \cite{szp}.

In order to prove Theorem \ref{goodeqbalanced} we will need the following general result.

\begin{thm}\label{imed2} Let $P$ be an ordered set and let
$x,y,z$ be three distinct elements such that $x<z$ and $y$ is
incomparable to both $x$ and $z$. Suppose that $(y,z)\in crit(P)$
and let $Q=P\vee\{(y,z)\}$. Then:
\begin{equation}\label{eq1}
{\mathbb{P}}_{Q}(x<y) < {\mathbb{P}}_{P}(x<y) \leq \frac{{2
\mathbb{P}}_{Q}(x<y)}{1+ {\mathbb{P}}_{Q}(x<y)}.
\end{equation}

\end{thm}

\vspace{.1in}

We should mention here that $\displaystyle \frac{{2
\mathbb{P}}_{Q}(x<y)}{1+ {\mathbb{P}}_{Q}(x<y)} \leq 1$ for every
$x,y$ and that $\displaystyle \frac{{2 \mathbb{P}}_{Q}(x<y)}{1+
{\mathbb{P}}_{Q}(x<y)} \leq \frac{2}{3}$ if and only if
$\displaystyle {\mathbb{P}}_{Q}(x<y) \leq \frac{1}{2}$. The second
inequality of (\ref{eq1}) above is tight as demonstrated by the
example (a) depicted in Figure \ref{alpha}. Moreover, if $(y,z)\not \in crit(P)$, then there exist $y'\leq y$ and
$z\leq z'$ such that $(y',z')\in crit(P)$. Obviously, $y'$ is incomparable to $x$ and $z'$.

\begin{proof} (Of Theorem \ref{imed2}) Denote by $\mathcal{L}(P)$ the set of linear extensions of $P$ and let $a_{1} = |\{L\in \mathcal{L}(P) :
x<_{L} y<_{L} z\}|$, $a_{2} = |\{L\in \mathcal{L}(P) : y<_{L} x\}|$
and $b = |\{L\in \mathcal{L}(P) : z<_{L} y\}|$. Then

$${\mathbb{P}}_{P}(x<y) =\frac{b+a_{1}}{b+a_{1}+a_{2}}\qquad {\rm and
}\qquad {\mathbb{P}}_{Q}(x<y) = \frac{a_{1}}{a_{1}+a_{2}}.$$

Proving the first inequality of Theorem \ref{imed2} amounts to
proving
$$ \frac{a_{1}}{a_{1}+a_{2}} < \frac{b+a_{1}}{b+a_{1}+a_{2}}.$$
which is true. Proving the second inequality amounts to proving that
$b\leq a_{1}$ since
$$\frac{{ \mathbb{P}}_{Q}(x<y)}{1+ {\mathbb{P}}_{Q}(x<y)} =
\frac{\frac{a_{1}}{a_{1}+a_{2}}}{1+\frac{a_{1}}{a_{1}+a_{2}}}=\frac{a_{1}}{2a_{1}+a_{2}}.$$
and
$$\frac{b+a_{1}}{b+a_{1}+a_{2}}\leq \frac{2a_{1}}{2a_{1}+a_{2}} \Leftrightarrow b\leq
a_{1}.$$

This last inequality is a consequence of Lemma \ref{l3}. Indeed,
there exists an injection from the set of linear extensions in which
$z<y$ (and $x<y$) to the set in which $y<z$ and $x<y$, obtained by
swapping the positions of $y$ and $z$ in the linear extension. It
follows that $b\leq a_{1}$.
\end{proof}

We now proceed to the proof of Theorem \ref{goodeqbalanced}.
\begin{proof} We consider the three cases separately.
\begin{enumerate}[(i)]
  \item Let $z\in V$ be such that $x<z$, $x\nsim y \nsim z$ and $\{x,y\}$ is autonomous in $P\setminus \{z\}$. Firstly, $z$ is an upper cover
of $x$. To prove this let $t$ be such that $x<t<z$. Then $t>y$ since $\{x,y\}$ is autonomous for $P\setminus \{z\}$. But then $y<z$, contradicting our assumption.

Secondly $(y,z)\in crit(P)$. To prove this let $u<y$. Then $u<x$ since $\{x,y\}$ is autonomous for $P\setminus \{z\}$. By
transitivity we get $u<z$. Now let $z<v$. Again by transitivity we have $x<v$. Hence, $y<v$ since $\{x,y\}$ is autonomous for
$P\setminus \{z\}$.

Consider $Q:=P\vee\{(y,z)\}$ and notice that $(x,y)$ and $(y,x)$ are critical in $Q$. It follows that ${\mathbb{P}}_{Q}(x<y)=\frac{1}{2}$. From Theorem \ref{imed2} we deduce that $(x,y)$ is balanced in $P$.
\item Let $z,t\in V$ be such that $x<z$, $y<t$, $y\nsim z$, $x\nsim t$ and $\{x,y\}$ is
autonomous for $P\setminus \{z,t\}$. Similar arguments as in $(i)$ yield that $z$ is an upper
cover of $x$, $t$ is an upper cover of $y$ and $\{(y,z),(x,t)\}\subseteq crit(P)$. Consider $Q:=P\vee\{(y,z)\}$ and
observe that $(y,x)\in crit(Q)$ and therefore ${\mathbb{P}}_{Q}(x<y)< \frac{1}{2}$ (Lemma \ref{l3}). Moreover,
$\{x,y\}$ is autonomous for $Q\setminus \{t\}$ which implies that $(x,y)$ is balanced in $Q$, and hence ${\mathbb{P}}_{Q}(x<y)\geq \frac{1}{3}$ (this is because $Q$ satisfies condition $(i)$ of Theorem \ref{goodeqbalanced}). Apply Theorem~\ref{imed2}.
\item Let $z,t\in V$ be such that $t<x<z$, $y$ is incomparable to both $t$ and $z$, and $\{x,y\}$ is autonomous for $P\setminus
\{z,t\}$. Similar arguments as in $(i)$ yield that $z$ is an upper cover of $x$, $t$ is a lower cover of $x$ and
$\{(t,y),(y,z)\}\subseteq crit(P)$. Consider $Q:=P\vee\{(y,z)\}$ and observe that $(y,x)\in crit(Q)$ and therefore
${\mathbb{P}}_{Q}(x<y)< \frac{1}{2}$. Moreover, $\{x,y\}$ is autonomous for $Q\setminus \{t\}$ which implies that $(x,y)$ is
balanced in $Q$ and hence ${\mathbb{P}}_{Q}(x<y)\geq \frac{1}{3}$. Apply Theorem \ref{imed2}.
\end{enumerate}
\end{proof}

\section{Proof of Theorem \ref{forest}}

Before getting to the proof of Theorem \ref{forest} we will need few definitions and preliminary results.

A \emph{fence} (of \emph{length} $n$) is any order isomorphic to the order defined on $\{f_0,..., f_n\}$, $n\geq 0$, where the elements with even subscript are minimal, the elements with odd subscript are maximal (or vice versa), and elements $f_i$ and $f_j$ are comparable if and only if $i=j$ or $|i-j|=1$.

A \emph{crown} (of \emph{length} $n$) is any order isomorphic to the order defined on $\{c_1,...,c_{2n}\}$, $n\geq 2$, where the elements with even subscript are minimal, the elements with odd subscript are maximal and elements $c_i$ and $c_j$ are comparable if and only if $i=j$ or $|i-j|=1$ or $i=1$ and $j=2n$.

A \emph{diamond} is any order isomorphic to the order defined on $\{d_1,d_2,d_3,d_4\}$ where $d_1<d_2<d_4$ and $d_1<d_3<d_4$ are the only cover relations among these elements.

The ordered set $P=(V,\leq)$ is \emph{crown-free}, if either $P$ has no subset isomorphic to a crown of length $n\geq 2$ or $P$ has a subset $\{c_1,c_2,c_3,c_4\}$ isomorphic to a crown of length 2 and there is an element $z\in V$ such that $c_2<z<c_1$ and $c_4<z<c_3$. We also say that $P$ is \emph{diamond-free} if there is no subset isomorphic to a diamond.

\begin{lem}\label{lem:forest2} Let $P=(V,\leq)$ be an ordered set which is crown-free and diamond-free. If $P$ contains a fence of length $n$, then $P$ contains a fence of length $n$ whose minimal elements are minimal in $P$ and whose maximal elements are maximal in $P$.
\end{lem}
\begin{proof}Let $F:=\{f_0,..., f_n\}$, $n\geq 0$, be a fence of length $n$ and let $f_i$ be a minimal element of $F$. If $f_i$ is not minimal in $P$, then let $f<f_i$ be a minimal element in $P$. Since $P$ is crown-free and diamond-free, $f$ is incomparable to all elements of $F\setminus \{f_i\}$ except the upper cover(s) of $f_i$ in $F$. Hence $(F\setminus\{f_i\})\cup \{f\}$ is a fence of length $n$.
\end{proof}

\begin{lem}\label{lem:forest3}
 Let $P=(V,\leq)$ be an ordered set which is crown-free and diamond-free, $x\in V$, and let $F:=\{x=f_0, f_1,..., f_n\}$, $n\geq 2$, be a fence of maximum length among those fences starting at $x$ and assume that $f_n$ is minimal in $F$. Then
 \begin{enumerate}[(i)]
   \item $U(f_{n-2})\cap U(f_n)$ has a unique minimal element and this minimal element is less or equal to $f_{n-1}$.
   \item If $m_{n-2,n}$ is the unique minimal element of $U(f_{n-2})\cap U(f_n)$, then every element $f$ such that $f_n\leq f <m_{n-2,n}$ has a unique upper cover and this upper cover is comparable to $m_{n-2,n}$. In particular, every element larger or equal to $f$ is comparable to $m_{n-2,n}$.
 \end{enumerate}
\end{lem}
\begin{proof}$(i)$ Suppose that $U(f_{n-2})\cap U(f_n)$ has two distinct minimal elements $y_1$ and $y_2$. Then $\{f_{n-2},f_n,y_1,y_2\}$ would be a crown in $P$. Say $m_{n-2,n}$ is the unique minimal element of $U(f_{n-2})\cap U(f_n)$. Then $m_{n-2,n}\leq f_{n-1}$ because otherwise $m_{n-2,n}\nsim f_{n-1}$ and hence $\{f_{n-2},f_n,m_{n-2,n},f_{n-1}\}$ would be a crown in $P$.\\
$(ii)$ Let $f$ be such that $f_n\leq f <m_{n-2,n}$ and $t$ be an upper of $f$. We assume that $t\nsim m_{n-2,n}$ and we will argue to a contradiction. We will prove that $F\,':=F\cup \{t\}$ is a fence. Then $F\,'$ is a fence that starts at $x$ and is of length larger than that of $F$ and this is a contradiction.  We start by proving that $t$ is incomparable to both $f_{n-2}$ and $f_{n-1}$. Indeed, if not, then $\{f_{n-2},f,t,m_{n-2,n}\}$ would be a crown in $P$ or $\{f,t,m_{n-2,n},f_{n-1},\}$ would be a diamond in $P$ which is not possible. Now suppose there exists $0\leq l \leq n-3$ such that $t\sim f_l$. Then $f_l <t$ (indeed by assumption $f_n<t$ and $f_n$ is incomparable to all elements of $\{x=f_0, f_1,..., f_{n-2}\}$ hence $t\nless f_l$). Choose $0\leq l \leq n-3$  maximal such that $f_l<t$. If $f_l$ is minimal in $F$, then the set $\{f_{l},...,f_{n},t\}$ is a crown in $P$. Else if $f_l$ is maximal in $F$, then the set $\{f_{l+1},...,f_{n},t\}$ is a crown in $P$. This is a contradiction. Hence we have proved that $t$ is comparable to $m_{n-2,n}$, that is $t\leq m_{n-2,n}$ (this is because $f<m_{n-2,n}$ and $t$ is an upper cover of $f$). From our assumption that $P$ is diamond-free we deduce that $\{u : f\leq u\leq m_{n-2,n}\}$ is a chain. It follows then that the set of upper covers of $f$ is a chain and therefore $f$ has a unique upper cover. Finally we prove that if $t'\geq f$, then $t'\sim m_{n-2,n}$. If $m_{n-2,n}\leq t'$, there is nothing to prove. Next we suppose that $m_{n-2,n}\nleq t'$. Let $f'$ be the largest element verifying $f_n\leq f'<m_{n-2,n}$ and $f'<t'$. It follows from our previous discussion that $f'$ has a unique upper cover and that this upper cover is comparable to $m_{n-2,n}$. Hence, $t'<m_{n-2,n}$ and we are done. This completes the proof of the lemma
\end{proof}

The following corollary gives a characterization of ordered sets whose cover graph is a forest

\begin{cor}\label{lem:forest1} Let $P=(V,\leq)$ be an ordered set. The cover graph of $P$ is a forest if and only if $P$ is crown-free and diamond-free.
\end{cor}
\begin{proof}Clearly, if the cover graph of $P$ is a forest, then $P$ is crown-free and diamond-free. For the converse assume $P$ is crown-free and diamond-free
and let $F=\{f_0,..., f_n\}$, $n\geq 0$, be a fence of maximum length in $P$. It follows from Lemma \ref{lem:forest2} that we can assume that the minimal elements of $F$ are minimal in $P$ and the maximal elements of $F$ are maximal in $P$. By duality we may assume without loss of generality that $f_n$ is minimal in $P$. We claim that $f_n$ has a unique upper cover. If $n\leq 1$, then $P$ is a disjoint sum of chains and we are done. Else if $n\geq2$, then our claim follows from $(ii)$ of Lemma \ref{lem:forest3} with $f=f_n$. Now consider the ordered set $P\setminus \{f_n\}$. From our assumption that $P$ is crown-free and diamond-free it follows that $P\setminus \{f_n\}$ is also crown-free and diamond-free. An induction argument on the number of elements of $P$ shows that the cover graph of $P$ is a forest.
\end{proof}

\begin{lem}\label{f0nochain} Let $P=(V,\leq)$ be an ordered set which is not a chain and whose cover graph is a tree, $x\in V$, and let $F:=\{x=f_0, f_1,..., f_n\}$, $n\geq 2$, be a fence of maximum length among those fences starting at $x$ and assume that $f_n$ is minimal in $F$. If $U(f_n)$ is not a chain, then either $P$ has very good pair in $U(f_n)$ or there exists a fence $F\,':=(F\setminus\{f_{n-1},f_n\})\cup \{f'_{n-1},f'_n\}$ such that $f_{n-2}<f'_{n-1}>f'_n$, $f'_n$ is minimal in $P$ and $U(f'_n)$ is a chain.
\end{lem}
\begin{proof} We recall that $U(f_{n-2})\cap U(f_n)$ has a unique minimal element $m_{n-2,n}$ and $m_{n-2,n}\leq f_{n-1}$ ($(i)$ of Lemma \ref{lem:forest3}).\\
\textbf{Claim 1:} $\{t : f_n< t \mbox{ and } t\nsim m_{n-2,n}\}=\varnothing$.\\
\emph{Proof of Claim 1:} Follows from $(ii)$ of Lemma \ref{lem:forest3}.\\
\textbf{Claim 2:} $U(m_{n-2,n})$ is a chain if and only if $U(f_n)$ is a chain.\\
\emph{Proof of Claim 2:} Obviously, if $U(f_n)$ is a chain, then $U(m_{n-2,n})$ is also a chain. Now suppose that $U(m_{n-2,n})$ is a chain. From Claim 1 we deduce that in order to prove $U(f_n)$ is a chain it is enough to prove that the set $\{x : f_n\leq t\leq m_{n-2,n}\}$ is a chain. This is true since $P$ is diamond-free. This completes the proof of claim 2.\\

Suppose that $U(f_n)$ is not a chain. It follows from Claim 2 that $U(m_{n-2,n})$ is not a chain. Since $P$ is diamond-free $U(m_{n-2,n})$ has at least two maximal elements (in $P$) and every element of $U(m_{n-2,n})$ has a unique lower cover comparable to $m_{n-2,n}$. Set
\[T:=\{y\in U(m_{n-2,n}) : y \mbox{ has a lower cover } z \mbox{ such that } z\nsim m_{n-2,n}\}.\]
If $T=\varnothing$, then the lower covers of every element $y\in U(m_{n-2,n})$ are comparable to $m_{n-2,n}$. Hence every element $y\in U(m_{n-2,n})$ has a unique lower cover. It follows then that any two distinct maximal elements $a$ and $b$ of $U(m_{n-2,n})$ verify $D(a)\setminus D(b)$ and $D(b)\setminus D(a)$ are chains and therefore the pair $(a,b)$ is a very good pair and we are done. Else if $T\neq \varnothing$, then let $y$ be a maximal element of $T$. It follows that the lower covers of every element of $U(y)$ must be comparable to $y$. Furthermore, and since $P$ is diamond-free, every element of $U(y)$ has a unique lower cover. Now assume that $U(y)$ is not a chain. Then $U(y)$ has at least two maximal elements (this is because $P$ is diamond-free).  Clearly any two such elements of $U(y)$ form a very good pair and we are done.

For the remainder of the proof of the lemma we assume that $U(y)$ is a chain. Let $z$ be a lower cover of $y$ such that $z\nsim m_{n-2,n}$. In particular $z\not \in \{f_{n-2},f_n\}$.\\
\textbf{Claim 3:} For all $z'\leq z$, $z'$ is incomparable to all elements of $\{m_{n-2,n}\} \cup D(m_{n-2,n})$.\\ 
\emph{Proof of Claim 3:} Suppose there exists $u\in \{m_{n-2,n}\} \cup D(m_{n-2,n})$ and $u\sim z'$. If $z'<u$, then it follows from our assumption $z\nsim m_{n-2,n}$ that $z\neq z'$ and hence $\{z',m_{n-2,n},z,y\}$ is a diamond in $P$. Else if $u<z'$, then it follows from our assumption $z'\leq z$ and $z\nsim m_{n-2,n}$ that $u\neq m_{n-2,n}$ and hence $\{u,m_{n-2,n},z,y\}$ is a diamond in $P$. In both cases we obtain a contradiction. This completes the proof of Claim 3.\\
\textbf{Claim 4:} For all $z'\leq z$, $F\,':=\{x=f_0, f_1,..., f_{n-2},y,z'\}$, $n\geq 2$, is a fence of maximum length among those fences starting at $x$.\\
\emph{Proof of Claim 4:} From our assumption that $F$ is fence follows that $F\setminus \{f_{n-1},f_n\}$ is fence. Hence in order to prove Claim 4 all we have to prove is that $y$ is incomparable to all elements of $F\,'\setminus \{f_{n-2},y,z'\}$ and $z'$ is incomparable to all elements of $F\,'\setminus \{y,z'\}$. From our assumption that $P$ is crown-free and diamond-free follows easily that $y$ is incomparable to all elements of $F\,'\setminus \{f_{n-2},y,z'\}$. We now prove that $z'$ is incomparable to all elements of $F'\setminus \{y,z'\}$. Suppose there exists $0\leq l \leq n-2$ such that $z'\sim f_l$. Then $l\neq n-2$ (follows from Claim 3) and $z'<f_l$ (this is because $z'<y$ and $y$ is incomparable to all elements of $\{x=f_0, f_1,..., f_{n-2}\}$ and hence $f_l \nless z'$). Choose $0\leq l \leq n-3$  maximal such that $z'<f_l$. If $f_l$ is minimal in $F$, then the set $\{z',f_{l+1},...,f_{n-2},y\}$ is a crown in $P$. Else if $f_l$ is maximal in $F$, then the set $\{z',f_{l},...,f_{n-2},y\}$ is a crown in $P$. This is a contradiction. The proof of Claim 4 is now complete.\\
\textbf{Claim 5:} Let $t$ be such that $m_{n-2,n}\leq t<y$ and let $z'\leq z$. Then $t\nsim z'$.\\
\emph{Proof of Claim 5:} Suppose not. Then $z'<t$ (this is because $z'\leq z$ and $z\nsim m_{n-2,n}$) and hence $z'\neq z$ (this is because $z$ is a lower cover of $y$ and $t<y$). It follows then that $\{z',t,z,y\}$ is a diamond in $P$ which is impossible. This completes the proof of Claim~5.\\
\textbf{Claim 6:} For every $z'\leq z$, if $t> z'$, then $t$ is comparable to $y$.\\
\emph{Proof of Claim 6:} It follows from Claim 4 that $F\,':=\{x=f_0, f_1,..., f_{n-2},y,z'\}$, $n\geq 2$, is a fence of maximum length among those fences starting at $x$. It follows from $(i)$ of Lemma \ref{lem:forest3} applied to $F'$ that the smallest element of $U(z')\cap U(f_{n-2})$ must be less or equal to $y$. Claims 3 and 5 imply that $y$ is the smallest element of $U(z')\cap U(f_{n-2})$.  Applying $(ii)$ of Lemma \ref{lem:forest3} to $F\,'$ with $f=z'$ gives the required conclusion. The proof of Claim 6 is now complete.\\

Let $z'\leq z$ and $t\geq z'$. From Claim 6 we deduce that $t\sim y$. Since $P$ is diamond-free $\{t : z'\leq t\leq y\}$ must be a chain. It follows from our assumption $U(y)$ is a chain that $U(z')$ is a chain. It follows from Claim 4 that 
$F\,'=\{x=f_0, f_1,...,f_{n-2},y,z'\}=(F\setminus\{f_{n-1},f_n\})\cup \{y,z'\}$ is a fence of maximum length among those fences starting at $x$. Choosing $z'$ to be minimal in $P$ it becomes now apparent that the fence $F\,'$ satisfies the required conditions of the lemma and we are done.
\end{proof}

\begin{cor}\label{minimal} Let $P=(V,\leq)$ be an ordered set which is not a chain and whose cover graph is a tree and let $F:=\{f_0, f_1,..., f_n\}$, $n\geq 2$, be a fence of maximum length in $P$. If $f_0$ and $f_n$ are minimal elements in $P$, then $P$ has a very good pair.
\end{cor}
\begin{proof} We notice at once that $F$ is a fence of maximum length among those fences that start at $f_0$, respectively that start at $f_n$. Hence, if $f_0$ or $f_n$ is a minimal element in $P$, and hence minimal in $F$, then Lemma \ref{f0nochain} applies. Assume that $f_0$ and $f_n$ are minimal elements in $P$. If $n=2$, then it follows from Claim 1 of the proof of Lemma \ref{f0nochain} and symmetry that $\{x : f_2< x \mbox{ and } x\nsim m_{0,2}\}=\varnothing$ where $m_{0,2}$ is the unique minimal element of $U(f_0)\cap U(f_2)$. Hence, $U(f_0)\setminus U(f_2)$ and $U(f_2)\setminus U(f_1)$ are chains proving that $(f_0,f_2)$ is a very good pair and we are done. Now assume $n\geq 4$. If $U(f_0)$ and $U(f_n)$ are chains, then $(f_0,f_n)$ is a very good pair and we are done. Suppose $U(f_n)$ is not a chain. Applying Lemma \ref{f0nochain} to the fence $F$ with $x=f_0$ we deduce that either $P$ has a very good pair in $U(f_n)$ or there exists a fence $F\,':=(F\setminus\{f_{n-1},f_n\})\cup \{f'_{n-1},f'_n\}$ (of maximum length) such that $f'_n$ is minimal in $P$, $f_{n-2}<f'_{n-1}>f'_n$ and $U(f'_n)$ is a chain. If $U(f_0)$ is a chain, then the pair $(f_0,f'_n)$ is a very good pair and we are done. Else if $U(f_0)$ is not a chain, then applying Lemma \ref{f0nochain} to the fence $F'$ with $x=f'_n$ we deduce that either $P$ has a very good pair in $U(f_0)$ or there exists a fence
$F\,'':=(F\, '\setminus\{f_{0},f_1\})\cup \{f'_{0},f'_1\}$ (of maximum length) such that $f'_0$ is minimal in $P$, $f'_{0}<f'_{1}>f_2$  and $U(f'_0)$ is a chain. It follows then that the pair $(f'_0,f'_n)$ is a very good pair and we are done.
\end{proof}
We now proceed to the proof of Theorem \ref{forest}.
\begin{proof}
Let $P=(V,\leq)$ be an ordered set not totally ordered and whose cover graph is a forest. If all connected components of $P$ are chains, then any two distinct minimal elements of $P$ form a very good pair. Otherwise $P$ has a connected component which is not a chain. Clearly, a very good pair in this connected component remains very good in $P$. Hence, we lose no generality by assuming that $P$ is connected, that is, its cover graph is tree.

Let $F:=\{f_0,f_1,...,f_n\}$, $n\geq2$, be a fence of maximum length in $P$. It follows from Lemma \ref{lem:forest2} that we may assume that all the $f_i$'s are minimal or maximal in $P$ and by duality we may assume without loss of generality that $f_0$ is a minimal element in $P$. It follows from Lemma \ref{f0nochain} that we can assume $U(f_0)$ to be a chain. By duality and symmetry it then follows that we can assume that either $D(f_n)$ is a chain if $f_n$ is maximal or $U(f_n)$ is a chain if $f_n$ is minimal. It follows from Corollary \ref{minimal} that we can assume $f_n$ to be maximal (hence $n$ is odd). We now define
\[\mathcal{F}:=\{x : \mbox{ there exist } 1\leq i,j\leq n \mbox{ with } |i-j|=1 \mbox{ such that } f_i\leq x \leq f_{j}\}\]
and
\[D:=\{x\in \mathcal{F} : \mbox{ there exists a fence } F_x \mbox{ of length at least 2 starting at } x \mbox{ so that } \mathcal{F}\cap F_x=\{x\}\}.\]
We consider two cases.\\
\noindent \textbf{Case 1:} $D\neq \varnothing$.\\
Let $x\in D$ and let $F_x=\{x=e_0,e_1,...,e_k\}$, $k\geq 2$, be a fence of maximum length at least 2 (among those fences starting at $x$ and satisfying $\mathcal{F}\cap F_x=\{x\}$). We notice at once that $f_0 \nsim e_k\nsim ,f_n$ (this follows easily from our assumption that $P$ is crown-free and diamond-free). Assume that $e_k$ is minimal in $F_x$. It follows from Lemma \ref{f0nochain} applied to $P$ and the fence $F_x$ that if $U(e_k)$ is not a chain, then either $P$ has a very good pair or we can find a new fence $F\,'_x=\{e_0=x,...,e'_{k-1},e'_k\}$ such that $e'_k$ is minimal in $P$ and $U(e'_k)$ is a chain. If the former holds then we are done. Else if the latter holds, then it follows from $f_0$ is minimal in $P$ and  $e_0\nsim f_0$ that $e'_k\nsim f_0$. Hence, $(f_0,e'_k)$ is a very good pair. If $e_k$ is maximal in $F_x$, then it follows from Lemma \ref{f0nochain} applied to the dual of $P$ and to the dual of the fence $F_x$ that either $P$ has a very good pair or we can find a new fence $F\, ''_x=\{e_0=x,...,e''_{k-1},e''_k\}$, $k\geq 2$, such that $e''_k$ is maximal in $P$ and $D(e''_k)$ is a chain. If the former holds then we are done. Else if the latter holds, then it follows from $f_n$ is maximal in $P$ and  $e_k\nsim f_n$ that $e''_k\nsim f_n$. Hence $(f_n,e''_k)$ is a very good pair.\\
\textbf{Case 2:} $D=\varnothing$.\\
\textbf{Claim 1:} Let $x\in \mathcal{F}$. Then every element of $U(x)\setminus \mathcal{F}$ has a unique lower cover and this lower cover is comparable to $x$. Dually, every element of $D(x)\setminus \mathcal{F}$ has a unique upper cover and this upper cover is comparable to $x$.\\
\emph{Proof of Claim 1:} Suppose there exists $y\in U(x)\setminus \mathcal{F}$ that has two distinct lower covers $y_1$ and $y_2$ and note that $y_1\nsim y_2$. Then $y_1$ or $y_2$ is incomparable to $x$ because otherwise $x\not \in \{y_1,y_2\}$ and therefore $\{x,y_1,y_2,y\}$ is a diamond in $P$ which is not possible. Say $y_1$ is incomparable to $x$. Then $y_1\in \mathcal{F}$ because otherwise $\{x,y,y_1\}$ is a fence of length at least 2 starting at $x$ and verifying $\mathcal{F} \cap \{x,y,y_1\}=\{x\}$ contradicting $D=\varnothing$. Let $k',k$ be nonnegative integers such that $f_{k'}\leq y_1\leq f_{k}$ and $|k'-k|=1$. Since $x\in \mathcal{F}$ there are nonnegative integers $i$ and $j$ such that $f_i\leq x\leq f_j$ and $|i-j|=1$. If $y_1$ is comparable to $f_i$, that is $k'=i$, then $y_1\neq f_i$ (this is because $y_1$ is a lower cover of $y$ and $f_i<x<y$) and since $f_i$ is minimal in $P$ we have $f_i<y_1$. Hence, $\{f_i,x,y_1,y\}$ is a diamond in $P$. Else if $y_1$ is incomparable to $f_i$, then $\{y_1,f_k,...,f_i,y\}$ is a crown. In both cases we obtain a contradiction since $P$ is diamond-free and crown-free. This proves Claim 1.\\
\textbf{Claim 2:} If there exists $x\in \mathcal{F}$ such that $U(x)\setminus \mathcal{F}$ or $D(x)\setminus \mathcal{F}$ is not a chain, then $P$ has a very good pair.\\
\emph{Proof of Claim 2:} Let $x\in \mathcal{F}$ be such that $U(x)\setminus \mathcal{F}$ is not a chain. Since $P$ is diamond-free $(U(x)\setminus \mathcal{F})$ has at least two maximal elements and every element of $(U(x)\setminus \mathcal{F})$ has a unique lower cover comparable to $x$. It follows from Claim 1 of Case 2 that every element of $U(x)\setminus \mathcal{F}$ has a unique lower cover and that this lower cover is comparable to $x$. It becomes now apparent that any pair of distinct maximal elements of $U(x)\setminus \mathcal{F}$ is a very good pair and we are done.\\

It follows from Claim 2 that we can assume that for every element $x\in \mathcal{F}$ the sets $U(x)\setminus \mathcal{F}$ and $D(x)\setminus \mathcal{F}$ are chains.

For integers $0\leq i,j\leq n$ with $|i-j|=1$ and $i$~even,~set
\[ D_{i,j}:=\{x : f_i < x < f_{j} \mbox{ and there exists } t\not \in \mathcal{F} \mbox{ such that } t \mbox{ covers } x\}.\]
\noindent\textbf{Claim 3:} If $D_{0,1}\neq \varnothing$, then $P$ has a very good pair.\\
\emph{Proof of Claim 3:} Assume that $D_{0,1}\neq \varnothing$ and let $x$ be such that $f_0 < x < f_{1}$ and let $t\not \in \mathcal{F}$ be a cover of $x$. From $U(f_0)$ is a chain it follows that $t<x$ and hence $t$ is a lower cover of $x$. We claim that $t\nsim f_0$. If not, then it follows from our assumption that $f_0$ is minimal in $P$ that $f_0<t$. From $t\not \in \mathcal{F}$ it follows that $x$ is not an upper of $f_0$. Let $x'\in \mathcal{F}$ be an upper cover of $f_0$. But then the set $\{f_0,x',t,x\}$ is a diamond in $P$. A contradiction. Our claim is then proved. Now let $t'\leq t$ be a minimal element. It follows from Claim 1 of Case 2 and our assumption $U(x)\setminus \mathcal{F}$ is a chain that $U(t')$ is a chain. Hence the pair $(f_0,t')$ is a very good pair and we are done.\\

For the remainder of the proof we assume that $D_{0,1}=\varnothing$.\\

\noindent\textbf{Claim 4:} If $D_{2,1}\neq \varnothing$, then $P$ has a very good pair.\\
\emph{Proof of Claim 4:} We recall that $U(f_{0})\cap U(f_1)$ has a unique minimal element denoted $m_{0,2}$ and that $f_0<m_{0,2}\leq f_{1}$. Let $x\in D_{2,1}$ and notice that since $D_{0,1}=\varnothing$ we have $f_2< x< m_{0,2}$. Choose $x$ to be maximal in $D_{2,1}$. We argue on whether $x$ is a lower cover of $m_{0,2}$ or not. We first consider the case $x$ is a lower cover of $m_{0,2}$. Let $t$ be a cover of $x$ not in $\mathcal{F}$. Suppose $t$ is a lower cover of $x$ and let $t'\leq t$ be a minimal element in $P$. We claim that $(f_0,t')$ is a very good pair. Indeed, by assumption $U(f_0)$ is a chain and hence $U(f_0)\setminus U(t')$ is a chain. Moreover, it follows from the maximality of $x$ and Claim 1 of Case 2 that $U(t')\setminus U(f_0)$ is also a chain. Since $f_0$ and $t'$ are both minimal in $P$ our claim follows. Now suppose that $t$ is an upper cover of $x$ and let $t''\geq t$ be a maximal element in $P$. We claim that $(f_1,t'')$ is a very good pair. Indeed, $D(t'')\setminus D(f_1)=\{z: t\leq z<t''\}$ which is a chain (this follows from Claim 1 of Case 2 and our assumption that $D(x)\setminus \mathcal{F}$ is a chain). Moreover, $D(f_1) \setminus D(t'')=\{z: f_0\leq z<f_1\}$ which is also a chain (by assumption $D_{0,1}=\varnothing$). The required conclusion follows since $f_1$ and $t''$ are maximal in $P$. Now we consider the case $x$ is not a lower cover of $m_{2,1}$. From our choice of $x$ it follows that for all $u$ such that $x<u<m_{2,1}$ we have $u\not \in D_{2,1}$, that is, every cover of $u$ is in $\mathcal{F}$. From our assumption that $U(f_0)$ is a chain follows that $U(u)$ is a chain. Let $u$ be an upper cover of $x$ such that $x<u<m_{2,1}$. Then $D(u)=D(t)=\{x\}\cup D(x)$. Hence $(u,t)$ is a very good pair and we are done.

For the remainder of the proof we assume that $D_{2,1}=\varnothing$. \\
%

Now it becomes apparent that similar arguments as in the proof of Claim 4 lead to $P$ has a very good pair if $D_{2,3}\neq \varnothing$. Hence we may assume that $D_{2,3}= \varnothing$. Let $y_1$ and $y_2$ be two distinct lower covers of $m_{0,2}$ such that $f_0\leq y_1<m_{0,2}$  and $f_2\leq y_2<m_{0,2}$. We claim that $(y_1,y_2)$ is a very good pair if $y_2\neq f_2$, or $(f_0,f_1)$ is a very good pair if $y_2= f_2$. Indeed, $D(y_1)$ is a chain since $D_{0,1}=\varnothing$ and $D(y_2)$ is a chain since $D_{2,1}=\varnothing$ and $U(y_1)U(m_{0,2})\cup \{m_{0,2}\}$ is a chain since $U(f_0)$ is a chain. Moreover, if $y_2\neq f_2$, then $U(y_2)=U(m_{0,2})\cup \{m_{0,2}\}$ which is a chain, else if $y_2= f_2$, then $f_2$ is a lower cover of $m_{0,2}$ and $U(f_2)\setminus U(f_1)$ is a chain since by assumption $D_{2,3}= \varnothing$. This proves our claim and completes the proof of the theorem.
\end{proof}

\noindent{\bf Acknowledgement}\ :
The author thanks two anonymous referees for their careful reading of the manuscript and for their remarks and suggestions.

\bibliographystyle{amsplain}

\end{document}